\documentclass[twoside, 14pt]{article}
\usepackage{amsmath,amssymb,amsthm,mathrsfs}
\usepackage[margin=2.9cm]{geometry} 
\usepackage{lipsum}
\usepackage{titlesec,hyperref}
\usepackage{fancyhdr}
\usepackage[numbers,sort&compress]{natbib}
\usepackage{color}

\fancypagestyle{plain}
{
\fancyhf{}

}

\titleformat{\subsection}{\it}{\thesubsection.\enspace}{1.5pt}{}
\titleformat{\subsubsection}{\it}{\thesubsubsection.\enspace}{1.5pt}{}

\newtheorem{theo}{Theorem}[section]
\newtheorem{lemm}[theo]{Lemma}

\newtheorem{prop}[theo]{Proposition}
\newtheorem{rema}{Remark}[section]
\numberwithin{equation}{section}

\allowdisplaybreaks

\def\lm{\le}
\def\p{\partial}
\def\a{\alpha}

\def\th2{\frac{\theta}{2}}
\def\dive{\mathop{\rm div}\nolimits}
\def\r{\varrho}
\def\lg{\geqslant}
\def\R{\mathbb{R}}

\begin{document}
\title{The Optimal Decay Rate of Strong Solution for the Compressible Navier-Stokes Equations
        with Large Initial Data  \hspace{-4mm}}
\author{Jincheng Gao$^\dag$ \quad Zhengzhen Wei$^\ddag$ \quad Zheng-an Yao$^\sharp$ \\[10pt]
\small {School of Mathematics, Sun Yat-Sen University,}\\
\small {510275, Guangzhou, P. R. China}\\[5pt]
}

\footnotetext{Email: \it $^\dag$gaojc1998@163.com,
\it $^\ddag$weizhzh5@mail2.sysu.edu.cn,
\it $^\sharp$mcsyao@mail.sysu.edu.cn}
\date{}

\maketitle

\begin{abstract}

In recent paper \cite{he-huang-wang}, it is shown that the upper decay rate of
global solution of compressible Navier-Stokes(CNS) equations converging to
constant equilibrium state $(1, 0)$ in $H^1-$norm is $(1+t)^{-\frac34(\frac{2}{p}-1)}$
when the initial data is large and belongs to $H^2(\R^3) \cap L^p(\R^3) (p\in[1,2))$.
Thus, the first result in this paper is devoted to showing that the upper decay rate of the
first order spatial derivative converging to zero in $H^1-$norm is $(1+t)^{-\frac32(\frac1p-\frac12)-\frac12}$.
For the case of $p=1$, the lower bound of decay rate for the global solution of CNS equations converging to
constant equilibrium state $(1, 0)$ in $L^2-$norm is $(1+t)^{-\frac{3}{4}}$ if the initial data
satisfies some low frequency assumption additionally.
In other words, the optimal decay rate for the global solution of CNS equations converging to constant
equilibrium state in $L^2-$norm is $(1+t)^{-\frac{3}{4}}$ although the associated initial data is large.

\vspace*{5pt}

\vspace*{5pt}
\end{abstract}


\section{Introduction}
\quad
In this paper, we are concerned with the upper and lower bounds of decay rate
for large solution to the three dimensional barotropic compressible Navier-Stokes(CNS) equations:
\begin{equation}\label{cns}
  	\left\{\begin{aligned}
  &\partial_t\rho+\dive(\rho u)=0,\\
   &\partial_t(\rho u)+\dive(\rho u\otimes u)-\mu\Delta u-(\mu+\lambda)\nabla\dive u+\nabla P=0,\\
 & \underset{|x|\rightarrow\infty}{\lim}(\rho-1, u)(x, t)=(0,0),
  	\end{aligned}\right.
  \end{equation}
where $(x, t)\in \mathbb{R}^3 \times \mathbb{R}^+$.
The unknown functions $\rho, u=(u_1, u_2, u_3)$ and $P$ represent
the density, velocity and pressure respectively.
The pressure $P$ is given by smooth function $P=P(\rho)=\rho^\gamma$ with the adiabatic exponent $\gamma \ge 1$.
The constants $\mu$ and $\lambda$ are the viscosity coefficients,
which satisfies the following conditions:
$\mu>0$, $2\mu+3\lambda \ge0$.
To complete system \eqref{cns}, the initial data is given by
\begin{equation*}
\left.(\rho, u)(x, t)\right|_{t=0}=(\rho_0(x), u_0(x)).
\end{equation*}

Compressible Navier-Stokes equations \eqref{cns} govern the motion of a compressible viscous barotropic fluid, and there are many literatures on the compressible Navier-Stokes equations because of its physical importance and mathematical challenges. Here we review some results which are related to well-posedness.
When the initial data is away from vacuum, Nash \cite{nash} proved the local well-posedness
for the compressible Navier-Stokes equations. Matsumura and Nishida \cite {matsumura} first established the global existence with the small initial data in $H^3-$framework. Later, Valli \cite{valli} and Kawashita \cite{kawa} obtained
the global existence with the small initial data in $H^2-$framework. Recently, Huang, Li and Xin \cite{Huang-li-xin} proved the global existence and uniqueness of system \eqref{cns} with the density containing vacuum in the condition that the initial energy is small. For further results about the well-posedness,
we refer the readers to \cite{guo-wang, li hl} and the references therein.

The decay problem has been one of main interests in mathematical fluid dynamics, there are many interesting
work has been obtained.
The optimal decay rate of strong solution was addressed in whole space firstly by Matsumura and Nishida \cite{nishida},
and the optimal $L^p(p\ge 2)$ decay rate is established by Ponce \cite{ponce}.
The authors obtained the optimal decay rate for Navier-Stokes system with an external potential force in series of papers \cite{duan1, duan2, u-y-z}. By assuming the initial perturbation is bounded in $\dot{H}^{-s}$ rather than $L^1$,
Guo and Wang \cite{guo-wang} built
the time decay rate for the solution of system \eqref{cns} by using a general energy method. It should be emphasized that their method in \cite{guo-wang} can be used to many other kinds of equations, such as Boltzmann equation, as well as some related fluid models. Many other results for the decay problem for the isentropic or non-isentropic Navier-Stokes equations, one can refer to \cite{wang hq, li-zhang1, liu-wang, wang yj} and the references therein.

However, the most of above decay results are established under the condition that the initial data is a small perturbation of constant equilibrium state. A interesting question is what may happen about the large time behavior of global strong solution with general initial data. Very recently, He, Huang and Wang \cite{he-huang-wang} proved global stability of large solution to the system \eqref{cns}. Let us give a short review of their work. By using some techniques about the blow-up criterion come from \cite{Huang-li-xin, Huang-li, sun-wang1, sun-wang2, wen}, and assuming the density is bounded uniformly in time in $C^{\alpha}$ with $\alpha$ arbitrarily small, that is $\sup_{t\ge0}\|\rho(t)\|_{C^\a}\lm M$, they obtained uniform-in-time bounds for the global solution. This allows them to improve the dissipation inequality, thus they could apply Fourier splitting method(see \cite{schonbek1}) to obtain a new method for the convergence to the constant equilibrium. It should be mentioned that they also constructed global large solution with a class of initial data which is far away from equilibrium.
Specifically, they established upper decay rate
\begin{equation}\label{Decay-large}
\|(\rho-1)(t)\|_{H^1}+\|u(t)\|_{H^1}\lm C(1+t)^{-\frac34(\frac{2}{p}-1)}.
\end{equation}
Here the initial data $(\rho_0-1,u_0 )\in L^{p}(\R^3)\cap H^2(\R^3)$ with $p\in[1,2]$.
As mentioned in \cite{he-huang-wang}(see Remark $1.2$), the decay rate \eqref{Decay-large},
which the solution itself converges to the constant equilibrium state $(1, 0)$,
is optimal compared to the heat equation.
However, the decay rate \eqref{Decay-large} shows that the first order spatial derivative of
solution converges to zero at the $L^2-$rate $(1+t)^{-\frac34(\frac{2}{p}-1)}$, which seems not optimal.
At the same time, the decay rate \eqref{Decay-large} does not provide any information
whether the second order spatial derivative of solution will converge to zero or not?\textit{
Thus, our first purpose is not only to establish optimal decay rate for the solution's first order spatial derivative,
but also prove that the second order spatial derivative of global solution will converge to zero.}

Finally, we address the lower bound of decay for the global strong solution investigated
in \cite{he-huang-wang} with large initial data.
For the case of incompressible flows, there are many mathematical results about decay to the
incompressible Navier-Stokes equations, for upper bound of decay rate to weak solution \cite{schonbek1},
for upper bound of decay rate with smooth initial data \cite{{duyi},{schonbek2},{Schonbek-Wiegner}},
for lower bound of decay rate \cite{{schonbek3},{Schonbek-JAMS}}, and both upper and lower bounds of decay for
the higher order spatial derivative \cite{titi}.
For the case of compressible flow, there are many results of lower bound of decay rate for the
CNS equations and related models,
such as CNS equations \cite{{Kagei-Kobayashi-2002},{li-zhang1}},
compressible Navier-Stokes-Poission equations \cite{{Li-Matsumura-Zhang},{Zhang-Li-Zhu}},
and compressible viscoelastic flows \cite{Hu-Wu}.
We point out that all these lower bounds of decay were obtained under the condition of small initial data.
\textit{Thus, our second purpose is to address lower bound of decay rate of global solution
investigated in \cite{he-huang-wang} with large initial data.}

Before state the main results of our paper, we have to introduce some notation.

 \textbf{Notation:} In this paper, we use $H^s(s\in\mathbb{R}^3)$ to denote the usual Sobolev space with norm $\|\cdot\|_{H^s}$ and $L^p(\mathbb{R}^3)$ to denote the usual $L^p$ space with norm $\|\cdot\|_{L^p}$. $\mathcal{F}(f):=\widehat{f}$ represents the usual Fourier transform of the function $f$.
 For the sake of simplicity, we write $\int f dx:=\int _{\mathbb{R}^3} f dx$
 and $\|(A, B)\|_X:=\| A \|_X+\| B\|_X$.
 The constant $C$ denotes the generic positive constant independent of time, and may change from line to line.

 First of all, we recall the following results obtained in \cite{he-huang-wang}, which will be used in this paper frequently.
 \begin{theo}(see \cite{he-huang-wang})\label{Thm1}
 Let $\mu>\frac12\lambda$, and $(\rho,u)$ be a global and smooth solution of \eqref{cns} with initial data $(\rho_0,u_0)$ where $\rho_0\ge c>0$. Suppose the admissible condition holds:
 $$ \partial_tu\mid_{t=0}=-u_0\cdot\nabla u_0+\frac{1}{\rho_0}Lu_0-\frac{1}{\rho_0}\nabla \rho_0^\gamma,$$
 where operator $L$ is defined by $Lu=-\dive(\mu\nabla u)-\nabla((\lambda+\mu)\dive u)$. Assume that $\r:=\rho-1$, and $\sup_{t\ge0}\|\rho(t)\|_{C^\a}\lm M$ for small $0<\a<1$. Then if $\r_0,u_0 \in L^{p}(\R^3)\cap H^2(\R^3)$ with $p\in[1,2]$, we have\\
(1)\textbf{(Lower bound of the density)}\\
 \quad There exists a positive constant $\underline{\rho}=\underline{\rho}(c,M)$
 such that for all $t\ge0$
 \begin{equation}\label{Lower}
 \rho(t)\ge\underline{\rho}.
  \end{equation}
(2)\textbf{(Uniform-in-time bounds for the regularity of the solution)}
 \begin{equation}\label{uniform-bound}
 \|\r\|^2_{L^\infty(H^2)}+\|u\|^2_{L^\infty(H^2)}+\int_0^\infty(\|\nabla\r\|^2_{H^1}+\|\nabla u\|^2_{H^2})d\tau\lm C(\underline{\rho},M,\|\r_0\|_{H^2},\|u_0\|_{H^2}).
 \end{equation}
(3)\textbf{(Decay estimate for the solution)}
 \begin{equation}\label{decay}
 \|u(t)\|_{H^1}+\|\r(t)\|_{H^1}\lm
 C(\underline{\rho},M,\|\r_0\|_{L^{p}\cap H^1},\|u_0\|_{L^{p}\cap H^2})(1+t)^{-\beta(p)},
 \end{equation}
 where $\beta(p)=\frac34(\frac{2}{p}-1)$.
 \end{theo}

 In this paper, we are not only to establish decay rate for the first and second order spatial derivatives
 of solution, but also give the lower bound of decay rate for the solution itself.
 Thus, we require the index $p$ in Theorem \ref{Thm1} satisfies $p\in [1, 2)$.
 Our first result can be stated as follows:

\begin{theo}\label{Thm2}
Define $\r :=\rho-1$,
suppose all the conditions in Theorem \ref{Thm1} hold on,
and let $(\rho, u)$ be the global solution of compressible Navier-Stokes equaitons \eqref{cns} in Theorem \ref{Thm1}.
Then, it holds on for all $t\lg T_1$
\begin{equation}\label{decay1}
\|\nabla \varrho(t)\|_{H^1}
+\|\nabla u(t)\|_{H^1}
+\|\partial_t \r(t)\|_{L^2}
+\|\partial_t  u(t)\|_{L^2}
\lm C(1+t)^{-\frac{3}{4}(\frac2p-1)-\frac{1}{2}}.
\end{equation}
Here $C$ is a constant independent of time,
and $T_1$ is a large constant given in Lemma \ref{lemma25}.
\end{theo}

\begin{rema}
Compared with decay rate \eqref{decay}, the advantage of decay rate \eqref{decay1} not only implies
that the second order spatial derivative of solution tends to zero, but also shows that
the first and second order spatial derivatives of solution converge to zero at the $L^2-$rate
$(1+t)^{-\frac{3}{4}(\frac2p-1)-\frac{1}{2}}, p \in [1, 2)$.
The decay rate for the first order spatial derivative of solution is optimal in the sense that
it coincides with the decay rate of solution to the heat equation.
\end{rema}

\begin{rema}
By the Sobolev interpolation inequality, it is shown that
the solution $(\rho, u)$ converges to the constant equilibrium state $(1, 0)$
at the $L^q(2\le q \le 6)-$rate $(1+t)^{-\frac{3}{4}(\frac{2}{p}-1)-\frac{3q-6}{4q}}, p \in [1, 2)$.
\end{rema}

Finally, we investigate the lower bound of decay rate for the density and velocity.
In order to make the upper bound of decay rate the same as the lower one, we take the index
$p=1$ in Theorem \ref{Thm1} specially.
Our second result can be stated as follows:

\begin{theo}\label{Thm3}
Let $p=1$, and suppose all the assumptions of Theorem \ref{Thm1} hold on. Denote $m_0:=\rho_0u_0$, assume that the Fourier transform $\mathcal{F}(\r_0, m_0)=(\widehat{\r_0}, \widehat{m_0})$ satisfies $
|\widehat{\r_0}|\lg c_0$, $\widehat{m_0}=0,0 \lm |\xi| \ll 1$, with $c_0>0$ a constant.
Then, the global solution $(\r,u)$ obtained in Theorem \ref{Thm1} has the decay rates for large time $t$
\begin{gather}\label{decay31}
c_3(1+t)^{-\frac34}\lm\|u(t)\|_{L^2}\lm C_1 (1+t)^{-\frac34};\\
\label{decay32}
c_3(1+t)^{-\frac34}\lm\|\r(t)\|_{L^2}\lm C_1 (1+t)^{-\frac34}.
\end{gather}
Here $c_3$ and $C_1$ are constants independent of time.
\end{theo}

\begin{rema}
The decay rates \eqref{decay31} and \eqref{decay32} imply that
the solution itself converges to
the constant equilibrium state $(1, 0)$ at the $L^2-$ rate
$(1+t)^{-\frac{3}{4}}$.
In other words, these decay rates obtained in \eqref{decay31} and \eqref{decay32}
are optimal, although the initial data for the CNS equations \eqref{cns} is large.
\end{rema}

Now we comment on the analysis in this paper.
First of all, we hope to establish the decay rate for the first and second order spatial derivatives of solution
for the compressible Navier-Stokes equations \eqref{cns} with large initial data.
Since the solution itself and its first order spatial derivative admit the same
$L^2-$rate $(1+t)^{-\frac34(\frac{2}{p}-1)}$, these quantities can be small enough essentially if the time is large.
Thus, we will take the strategy of the frame of small initial data(cf.\cite{nishida})
to establish the energy estimate:
\begin{equation}\label{energy1}
\frac{d}{dt}\mathcal{E}^2_1(t)+c_* (\|\nabla^2 u\|_{H^1}^2+\|\nabla^2 \r\|_{L^2}^2)
\lesssim Q(t) (\|\nabla^2 u\|_{H^1}^2+\|\nabla^2 \r\|_{L^2}^2)
+\|\nabla u\|_{L^{\infty}}\|\nabla^2\r\|^2_{L^2},
\end{equation}
where the energy norm $\mathcal{E}^2_1(t)$ is equivalent to $\|\nabla (\r, u)\|_{H^1}^2$,
and $Q(t)$ consists of some difficult terms, such as $\|\r\|_{L^\infty}$ and $\|\nabla(\r, u)\|_{L^3}$.
It is worth nothing that one can apply the Sobolev interpolation inequality to control these quantities
by the product of solution itself and the second order spatial derivative.
Since the latter one is uniform bounded with respect to time, $Q(t)$ is a small quantity
which appears as a prefactor in front of dissipation term $(\|\nabla^2 u\|_{H^1}^2+\|\nabla^2 \r\|_{L^2}^2)$,
which can be absorbed into the second term on the left hand side of inequality \eqref{energy1}.
On the other hand, the term $\|\nabla u\|_{L^{\infty}}$ can be controlled by product of
the first order spatial derivative of velocity
and dissipation term, see \eqref{231}.
Thus, the terms on the right hand side of \eqref{energy1} can be absorbed into
the second term on the left hand side of inequality \eqref{energy1}.

Secondly, we hope to perform the upper decay rate \eqref{decay1} by using the energy inequality
\eqref{energy1} and the Fourier splitting method by Schonbek \cite{schonbek1}.
Compared with incompressible flows(cf.\cite{{schonbek2},{Schonbek-Wiegner}}),
the dissipation of density is weaker than the one of velocity for the compressible Navier-Stokes equations.
To overcome this difficulty, our method here is to weaken the coefficient of velocity dissipation;
and hence, one part of the dissipation of density will play a role of damping term.
Thus, the application of Fourier splitting method helps us to obtain the decay rate \eqref{decay1},
see Lemma \ref{lemma25}.

Finally, we study the lower bound of decay rate for global solution of compressible Navier-Stokes
equations associated with large initial data for the case of $p=1$.
Since the decay rate \eqref{decay1} implies that these quantities will be small enough essentially
when the time is large.
It is worth nothing that the lower bound of decay rate for the linearized part
has been obtained in \cite{{Hu-Wu},{li-zhang1}} associated with large initial data.
Thus, let $U$  and $U_l$ be the solutions of nonlinear and linearized problem respectively.
Define the difference $U_\delta:= U-U_l$, it holds on
$
\|U\|_{L^2} \ge  \|U_l\|_{L^2}-\|U_\delta\|_{L^2}.
$
If the solutions $U_l$ and $U_\delta$ obey the assumptions:
$
\|U_l\|_{L^2} \ge C_{l} (1+t)^{-\alpha},
\quad \|U_\delta\|_{L^2} \le C_{\delta} (1+t)^{-\alpha}.
$
If $C_{\delta}$ is a small constant, then we have
$
\|U\|_{L^2} \ge \frac{1}{2}{C_{l}}(1+t)^{-\alpha}.
$
Indeed, the constant $C_{\delta}$ in our analysis depends on the quantity $\|(\r, u)(t)\|_{H^1}$,
which is small enough when the time is large.
All these lower and upper bounds of decay rates \eqref{decay1}, \eqref{decay31} and \eqref{decay32}
will be established in Section \ref{Main Theorems}.

\section{Proof of Main Theorems}\label{Main Theorems}

\quad
In this section, we will give the proof for the main Theorems \ref{Thm2} and \ref{Thm3}.
In subsection \ref{upper-decay}, we will show not only the second order spatial derivative of solution
tends to zero, but also the first and the second order spatial derivatives of solution converge to zero at
the $L^2-$rate $(1+t)^{-\frac{3}{4}(\frac2p-1)-\frac{1}{2}}$ with $p \in [1, 2)$.
In subsection \ref{lower-bound-decay}, one investigates the lower bound of decay rate for the solution
$(\r, u)$. This will show that the decay rate obtained in
Theorem \ref{Thm1} is optimal essentially for the case $p=1$.

\subsection{Upper Bound of Decay Rate}\label{upper-decay}
\quad
In this subsection, the content of our analysis is to give the proof for the Theorem \ref{Thm2}.
The analysis proceeds in several steps, which we will now detail.
Denoting $\r:=\rho-1$, we rewrite \eqref{cns} in the perturbation form as follows
\begin{equation}\label{cns1}
  	\left\{\begin{aligned}
  &\partial_t\varrho+\dive u=S_1,\\
   &\p_tu-\mu\Delta u-(\mu+\lambda)\nabla\dive u+P'(1)\nabla\r=S_2,\\
  	\end{aligned}\right.
  \end{equation}
  where the nonlinear terms $S_1$ and $S_2$ are defined by
  \begin{equation*}
  \left\{
  \begin{aligned}
  &S_1:=-\varrho\dive u-u\cdot\nabla\varrho,\\
  &S_2:=-u\cdot \nabla u-\frac{\r}{\r+1}[\mu\Delta u+(\mu+\lambda)\nabla\dive u]
                        -[\frac{P'(1+\r)}{1+\r}-\frac{P'(1)}{1}]\nabla\r.
  \end{aligned}
  \right.
 \end{equation*}

The first estimate in our scheme is to perform the estimate for the
first order spatial derivative of density and velocity as follows.

 \begin{lemm}\label{lemm1}
 Under the assumptions of Theorem \ref{Thm1}, the global solution $(\r, u)$ of Cauchy problem \eqref{cns1}
 has the estimate
 \begin{equation}\label{estimate1}
\begin{split}
&\frac{d}{dt}\int(\frac12|\nabla u|^2+\frac{P'(1)}{2}|\nabla\r|^2)dx+\mu\int|\nabla^2 u|^2dx+(\mu+\lambda)\int|\nabla\dive u|^2dx\\
&\lm C(\|\r\|^{\frac14}_{L^2}+\|\r \|_{H^1}+\|u \|_{H^1})
         (\|\nabla^2 u\|_{L^2}^2+\|\nabla^2 \r\|_{L^2}^2).
\end{split}
 \end{equation}
 Here $C$ is a constant independent of time.
 \end{lemm}
 \begin{proof}
  First, applying $\nabla$ operator to the second equation of \eqref{cns1}, we have
  \begin{equation}\label{215}
  \p_t(\nabla u)-\mu\Delta\nabla u-(\mu+\lambda)\nabla^2\dive u+P'(1)\nabla^2\r=\nabla S_2.
  \end{equation}
  Multiplying equation \eqref{215} by $\nabla u$ and integrating over $\R^3$, we get
  \begin{equation*}
  \frac{d}{dt}\frac12\int|\nabla u|^2dx+\mu\int|\nabla^2u|^2dx+(\mu+\lambda)\int|\nabla\dive u|^2dx+P'(1)\int\nabla^2\r\cdot\nabla udx=\int\nabla S_2\cdot\nabla udx,
  \end{equation*}
  which, integrating by part and applying H\"{o}lder inequality, yields directly
  \begin{equation*}
  \frac{d}{dt}\frac12\int|\nabla u|^2dx+\mu\int|\nabla^2u|^2dx+(\mu+\lambda)\int|\nabla\dive u|^2dx
  +P'(1)\int\nabla^2\r\cdot\nabla udx \lm \|S_2\|_{L^2}\|\nabla^2u\|_{L^2}.
 \end{equation*}
Applying the H\"{o}lder and Sobolev inequalities, we show that
\begin{equation}\label{216}
\|u\cdot \nabla u\|_{L^2}\le \|u\|_{L^3}\|\nabla u\|_{L^6}
\le C\|u\|_{H^1}\|\nabla^2 u\|_{L^2}.
\end{equation}
Using the lower bound of density \eqref{Lower}, it holds on
\begin{equation*}
\|\frac{\r}{\r+1}[\mu\Delta u+(\mu+\lambda)\nabla\dive u]\|_{L^2}
\le C\|\r\|_{L^\infty}\|\nabla^2 u\|_{L^2}.
\end{equation*}
The combination of Sobolev inequality and uniform estimate \eqref{uniform-bound} yields directly
\begin{equation}\label{4}
\|\r\|_{L^{\infty}}
\lm C\|\r\|^{\frac14}_{L^2}\|\nabla^2\r\|^{\frac34}_{L^2}
\lm C\|\r\|^{\frac14}_{L^2},
\end{equation}
and hence, it holds on
\begin{equation}\label{211}
\|\frac{\r}{\r+1}[\mu\Delta u+(\mu+\lambda)\nabla\dive u]\|_{L^2}
\le C\|\r\|^{\frac14}_{L^2}\|\nabla^2 u\|_{L^2}.
\end{equation}
Using the Taylor expression, H\"{o}lder and Sobolev inequalities, we have
\begin{equation}\label{212}
\|[\frac{P'(1+\r)}{1+\r}-\frac{P'(1)}{1}]\nabla\r\|_{L^2}
\le C\|\r\|_{L^3}\|\nabla\r\|_{L^6}
\le C\|\r\|_{H^1}\|\nabla^2 \r\|_{L^2}.
\end{equation}
Hence, the combination of estimates \eqref{216}, \eqref{211} and \eqref{212} implies directly
\begin{equation*}
\|S_2\|_{L^2}\|\nabla^2u\|_{L^2}
\leqslant C\|\r\|^{\frac14}_{L^2}\|\nabla^2 u\|^2_{L^2}
          +C(\|\r\|_{H^1}+\|u\|_{H^1})(\|\nabla^2\r\|^2_{L^2}+\|\nabla^2u\|^2_{L^2}),
\end{equation*}
which implies
\begin{equation}\label{214}
\begin{aligned}
&\frac{d}{dt}\frac12\int|\nabla u|^2dx+\mu\int|\nabla^2u|^2dx
  +(\mu+\lambda)\int|\nabla\dive u|^2dx
  +P'(1)\int\nabla^2\r\cdot\nabla udx\\
&\lm C(\|\r\|^{\frac14}_{L^2}+\|\r\|_{H^1}+\|u\|_{H^1})(\|\nabla^2\r\|^2_{L^2}+\|\nabla^2u\|^2_{L^2}).
\end{aligned}
\end{equation}

 Second, applying $\nabla$ operator to the first equation of \eqref{cns1}, we have
 \begin{equation*}
 \p_t(\nabla\r)+\nabla\dive u=\nabla S_1.
 \end{equation*}
 Multiplying the above equality by $P'(1)\nabla\r$ and integrating over $\R^3$, it follows that
 \begin{equation*}
 \frac{d}{dt}\frac{P'(1)}{2}\int|\nabla\r|^2dx+P'(1)\int\nabla\r\cdot\nabla\dive udx
 =P'(1)\int\nabla S_1\cdot\nabla \r dx,
 \end{equation*}
 which, integrating by parts and applying H\"{o}lder inequality, yields directly
  \begin{equation}\label{density}
  \frac{d}{dt}\frac{P'(1)}{2}\int|\nabla\r|^2dx-P'(1)\int\nabla^2\r\cdot\nabla udx
   \lm C\|S_1\|_{L^2}\|\nabla^2\r\|_{L^2}.
  \end{equation}
 Using H\"{o}lder and Sobolev inequalities, one may check that
 \begin{equation}\label{5}
 \|S_1\|_{L^2}\le \|\r \|_{L^3}\|{\rm div} u\|_{L^6}+\|u\|_{L^3}\|\nabla \r\|_{L^6}
 \le C(\|\r \|_{H^1}+\|u \|_{H^1})(\|\nabla^2 u\|_{L^2}+\|\nabla^2 \r\|_{L^2}).
 \end{equation}
 This and the inequality \eqref{density} give directly
 \begin{equation}\label{213}
  \frac{d}{dt}\frac{P'(1)}{2}\int|\nabla\r|^2dx-P'(1)\int\nabla^2\r\cdot\nabla udx
   \lm C(\|\r \|_{H^1}+\|u \|_{H^1})(\|\nabla^2 u\|_{L^2}^2+\|\nabla^2 \r\|_{L^2}^2).
  \end{equation}
Combining the estimates \eqref{214} and \eqref{213}, we deduce
   \begin{equation*}
   \begin{split}
   &\frac{d}{dt}\int(\frac12|\nabla u|^2+\frac{P'(1)}{2}|\nabla\r|^2)dx
   +\mu\int|\nabla^2 u|^2dx+(\mu+\lambda)\int|\nabla\dive u|^2dx\\
   &\lm C(\|\r\|^{\frac14}_{L^2}+\|\r \|_{H^1}+\|u \|_{H^1})
         (\|\nabla^2 u\|_{L^2}^2+\|\nabla^2 \r\|_{L^2}^2).
   \end{split}
   \end{equation*}
Therefore, we conclude the proof of this lemma.
   \end{proof}

The content of the next step is to establish the energy estimate for the second order spatial derivative of solution,
which can help us to achieve the decay rate for them.

\begin{lemm}\label{lemm2}
Under the assumptions of Theorem \ref{Thm1}, the global solution $(\r, u)$ of Cauchy problem \eqref{cns1}
 has the estimate
\begin{equation}\label{estimate2}
\begin{split}
&\frac{d}{dt}\int(\frac12|\nabla^2u|^2+\frac{P'(1)}{2}|\nabla^2\r|^2)dx
 +\mu\int|\nabla^3u|^2dx+(\mu+\lambda)\int|\nabla^2\dive u|^2dx\\
&\lm  C(\|u\|_{H^1}+\|\r\|^{\frac14}_{L^2}+\|u\|^{\frac14}_{L^2}
        +\|\nabla u\|^{\frac14}_{L^2})(\|\nabla^2 u\|^2_{H^1}+\|\nabla^2\r\|^2_{L^2}),
\end{split}
\end{equation}
where $C$ is a constant independent of time.
\end{lemm}
\begin{proof}
   First, applying $\nabla^2$ differential operator to the second equation of \eqref{cns1}, it holds on
   \begin{equation*}
   \p_t(\nabla^2u)-\mu\nabla^2\Delta u-(\mu+\lambda)\nabla^3\dive u+P'(1)\nabla^3\r=\nabla^2S_2.
   \end{equation*}
   Multiplying the above equality by $\nabla^2u$ and integrating over $\R^3$, we get
   \begin{equation*}\label{2orderu}
   \frac{d}{dt}\frac12\int|\nabla^2u|^2dx+\mu\int|\nabla^3u|^2dx+(\mu+\lambda)\int|\nabla^2\dive u|^2dx
   +P'(1)\int\nabla^3\r\cdot\nabla^2udx=\int\nabla^2 S_2\cdot\nabla^2udx.
   \end{equation*}
   Let us focus on the last term $\int\nabla^2 S_2\cdot\nabla^2udx$.
   The integration by part yields directly
   \begin{equation}\label{last}
   \int\nabla^2 S_2\cdot\nabla^2udx=-\int\nabla S_2\cdot\nabla^3udx.
   \end{equation}
   By routine checking, one may show that
   \begin{equation}\label{s2}
   \begin{split}
   \nabla S_2=
   &-\nabla u\cdot \nabla u-u\cdot \nabla(\nabla u)
     -\frac{\r}{1+\r}[\mu\nabla\Delta u+(\mu+\lambda)\nabla^2\dive u]\\
   & -[\frac{P'(1+\r)}{1+\r}-\frac{P'(1)}{1}]\nabla^2\r
     -\frac{\nabla\r}{(1+\r)^2}[\mu\Delta u+(\mu+\lambda)\nabla\dive u]\\
   &-\frac{P''(1+\r)(1+\r)-P'(1+\r)}{(1+\r)^2}\nabla\r\nabla\r.
   \end{split}
   \end{equation}
Observe that
\begin{equation*}
P'(1+\r)=\gamma(1+\r)^{\gamma-1}, \qquad P''(1+\r)=\gamma(\gamma-1)(1+\r)^{\gamma-2},
\end{equation*}
and hence, it holds on
\begin{equation}\label{pre}
\frac{P''(1+\r)(1+\r)-P'(1+\r)}{(1+\r)^2}=\gamma(\gamma-2)(1+\r)^{\gamma-3}, \qquad \frac{P'(1+\r)}{1+\r}=\gamma(1+\r)^{\gamma-2}.
\end{equation}
The combination of \eqref{s2} and \eqref{pre} yields directly
\begin{equation*}
\begin{aligned}
\|\nabla S_2 \|_{L^2}
\lm &C(\|\nabla u\|_{L^3}\|\nabla u\|_{L^6}
     +\|u\|_{L^3}\|\nabla^2 u\|_{L^6}
     +\|\r\|_{L^{\infty}}\|\nabla^3 u\|_{L^2})\\
     &+C(\|\r\|_{L^{\infty}}\|\nabla^2\r\|_{L^2}
     +\|\nabla\r\|_{L^3}\|\nabla^2u\|_{L^6}+\|\nabla\r\|_{L^3}\|\nabla\r\|_{L^6})\\
\lm &C(\|\r\|_{L^{\infty}}+\|u\|_{H^1}
       +\|\nabla\r\|_{L^3}+\|\nabla u\|_{L^3})(\|\nabla^2 u\|_{H^1}+\|\nabla^2\r\|_{L^2}).
\end{aligned}
\end{equation*}
By virtue of the Sobolev inequality and uniform estimate \eqref{uniform-bound}, it follows that
\begin{equation}\label{3}
\begin{aligned}
\|\r\|_{L^{\infty}}+\|\nabla\r\|_{L^3}+\|\nabla u\|_{L^3}
&\lm C(\|\r\|^{\frac14}_{L^2}\|\nabla^2\r\|^{\frac34}_{L^2}+\|\r\|^{\frac14}_{L^2}\|\nabla^2\r\|^{\frac34}_{L^2}
      +\|u\|^{\frac14}_{L^2}\|\nabla^2 u\|^{\frac34}_{L^2})\\
&\lm C(\|\r\|^{\frac14}_{L^2}+\|u\|^{\frac14}_{L^2}),
\end{aligned}
\end{equation}
and hence, we show that
\begin{equation}\label{221}
\|\nabla S_2 \|_{L^2}\le C(\|u\|_{H^1}+\|\r\|^{\frac14}_{L^2}+\|u\|^{\frac14}_{L^2})
                          (\|\nabla^2 u\|_{H^1}+\|\nabla^2\r\|_{L^2}).
\end{equation}
Thus, we conclude the following estimate
\begin{equation}\label{2order1}
\begin{split}
&\frac{d}{dt}\frac12\int|\nabla^2u|^2dx+\mu\int|\nabla^3u|^2dx+(\mu+\lambda)\int|\nabla^2\dive u|^2dx
 +P'(1)\int\nabla^3\r\cdot\nabla^2udx \\
&\lm C(\|u\|_{H^1}+\|\r\|^{\frac14}_{L^2}+\|u\|^{\frac14}_{L^2})(\|\nabla^2 u\|_{H^1}^2+\|\nabla^2\r\|_{L^2}^2).
\end{split}
\end{equation}

Applying $\nabla^2$ differential operator to the first equation of \eqref{cns1} implies
\begin{equation*}
\p_t(\nabla^2\r)+\nabla^2\dive u=\nabla^2S_1.
\end{equation*}
Multiplying the above equality by $P'(1)\nabla^2\r$ and integrating over $\R^3$, we obtain
\begin{equation}\label{2orderdensity}
\frac{d}{dt}\frac{P'(1)}{2}\int|\nabla^2\r|^2dx+P'(1)\int\nabla^2\dive u\cdot\nabla^2\r dx= P'(1)\int\nabla^2 S_1\cdot\nabla^2\r dx
\end{equation}
Recall that $S_1=-\r\dive u-u\cdot\nabla\r$, a straightforward computation shows that
\begin{equation*}
\nabla^2(\r\dive u)=\r\nabla^2\dive u+2\nabla\r\nabla\dive u+\nabla^2\r\dive u,
\end{equation*}
and hence, it follows that
\begin{equation}\label{1}
\begin{split}
&\|\nabla^2(\r\dive u)\|_{L^2}\|\nabla^2\r\|_{L^2}\\
&\lm (\|\r\|_{L^{\infty}}\|\nabla^2\dive u\|_{L^2}+\|\nabla\r\|_{L^3}\|\nabla\dive u\|_{L^6})\|\nabla^2\r\|_{L^2}
     +\|\dive u\|_{L^{\infty}}\|\nabla^2\r\|_{L^2}^2\\
&\lm C(\|\r\|_{L^{\infty}}+\|\nabla\r\|_{L^3})(\|\nabla^3u\|^2_{L^2}+\|\nabla^2\r\|^2_{L^2})+C\|\dive u\|_{L^{\infty}}\|\nabla^2\r\|^2_{L^2}.
\end{split}
\end{equation}
By routine checking, one may check that
\begin{equation*}
\nabla^2(u\cdot\nabla\r)=u\cdot\nabla(\nabla^2\r)+2\nabla u\cdot\nabla^2\r+\nabla^2u\cdot\nabla\r.
\end{equation*}
The integration by part yields directly
\begin{equation*}
\int u\cdot\nabla(\nabla^2\r)\cdot\nabla^2 \r dx
=\int u\cdot\nabla(\frac12|\nabla^2\r|^2)dx=-\frac12\int (\dive u)|\nabla^2\r|^2 dx,
\end{equation*}
and hence, we obtain
\begin{equation}\label{2}
\begin{split}
|\int\nabla^2(u\cdot\nabla\r)\cdot\nabla^2\r dx|
&\lm C(\|\nabla u\|_{L^{\infty}}\|\nabla^2\r\|_{L^2}+\|\nabla^2u\|_{L^6}\|\nabla\r\|_{L^3})\|\nabla^2\r\|_{L^2}\\
&\quad    +C\|\dive u\|_{L^{\infty}}\|\nabla^2\r\|^2_{L^2}\\
&\lm C\|\nabla u\|_{L^{\infty}}\|\nabla^2\r\|^2_{L^2}
     +C \|\nabla\r\|_{L^3} \|\nabla^3u\|_{L^2}\|\nabla^2\r\|_{L^2}.
\end{split}
\end{equation}
It follows from the estimates \eqref{1} and \eqref{2} that
\begin{equation*}
|\int\nabla^2 S_1\cdot\nabla^2\r dx |
\lm C(\|\nabla\r\|_{L^3}+\|\r\|_{L^{\infty}}+\|\nabla u\|_{L^{\infty}})\|\nabla^2 \r\|^2_{L^2}
    +C(\|\nabla\r\|_{L^3}+\|\r\|_{L^{\infty}})\|\nabla^3u\|^2_{L^2},
\end{equation*}
which, together with \eqref{3} and \eqref{2orderdensity}, implies directly
\begin{equation}\label{2oeder2}
\begin{split}
&\frac{d}{dt}\frac{P'(1)}{2}\int|\nabla^2\r|^2dx+P'(1)\int\nabla^2\dive u\cdot\nabla^2\r dx\\
&\lm C(\|\r\|^{\frac14}_{L^2}+\|\nabla u\|_{L^{\infty}})\|\nabla^2 \r\|^2_{L^2}
     +C\|\r\|^{\frac14}_{L^2}\|\nabla^3 u\|^2_{L^2}.
\end{split}
\end{equation}
The combination of \eqref{2order1} and \eqref{2oeder2} gives rise to
\begin{equation}\label{6}
\begin{split}
&\frac{d}{dt}[\int\frac12|\nabla^2u|^2dx+\frac{P'(1)}{2}\int|\nabla^2\r|^2dx]+\mu\int|\nabla^3u|^2dx+(\mu+\lambda)\int|\nabla^2\dive u|^2dx\\
&\lm C(\|u\|_{H^1}+\|\r\|^{\frac14}_{L^2}+\|u\|^{\frac14}_{L^2})
      (\|\nabla^2 u\|^2_{H^1}+\|\nabla^2\r\|^2_{L^2})
       +C\|\nabla u\|_{L^{\infty}}\|\nabla^2\r\|^2_{L^2}.
\end{split}
\end{equation}
It is worth nothing that $\|\nabla u\|_{L^{\infty}}\|\nabla^2\r\|^2_{L^2}$ is the delicate term,
which arises on the righthand side of inequality \eqref{6}.
Then, our method here is to control the prefactor $\|\nabla u\|_{L^{\infty}}$ in front of $\|\nabla^2\r\|^2_{L^2}$
by the product of energy term $\|\nabla u\|_{L^2}$ and dissipative term $\|\nabla^3 u\|_{L^2}$.
More precisely, one may show that
\begin{equation}\label{231}
\begin{aligned}
\|\nabla u\|_{L^{\infty}}\|\nabla^2\r\|^2_{L^2}
&\lm C\|\nabla u\|^{\frac14}_{L^2}\|\nabla^3u\|^{\frac34}_{L^2}
     \|\nabla^2\r\|^{\frac54}_{L^2}\|\nabla^2\r\|^{\frac34}_{L^2}\\
&\lm C\|\nabla u\|^{\frac14}_{L^2}(\|\nabla^3u\|^2_{L^2}+\|\nabla^2\r\|^2_{L^2}),
\end{aligned}
\end{equation}
where we have used the uniform estimate \eqref{uniform-bound} in the last inequality.
Then, substituting the estimate \eqref{231} into \eqref{6}, we conclude the proof of this lemma.
\end{proof}

In order to close the estimate, it is imperative to establish the dissipation estimate for $\nabla^2 \r$.
\begin{lemm}\label{lemm3}
Under the assumptions of Theorem \ref{Thm1}, the global solution $(\r, u)$ of Cauchy problem \eqref{cns1}
 has the estimate
\begin{equation}\label{estimate3}
\frac{d}{dt}\int\nabla u\cdot\nabla^2\r dx+\frac{7}{8}P'(1)\int|\nabla^2\r|^2 dx
\lm C\|\nabla^2u\|^2_{H^1}+C(\|\r\|^{\frac14}_{L^2}+\|\r \|_{H^1}+\|u \|_{H^1})\|\nabla^2\r\|^2_{L^2},
\end{equation}
where $C$ is a positive constant independent of time.
\end{lemm}
\begin{proof}
Multiplying the equation \eqref{215} by $\nabla^2\r$ and integrating over $\R^3$, we get
\begin{equation*}
\int\p_t(\nabla u)\cdot\nabla^2\r dx-\int[\mu\Delta\nabla u+(\mu+\lambda)\nabla^2\dive u]\cdot\nabla^2\r dx+P'(1)\int|\nabla^2\r|^2 dx=\int\nabla S_2\cdot\nabla^2\r dx.
\end{equation*}
Using the transport equation, that is the first equation of \eqref{cns1}, it holds on
\begin{equation*}
\begin{split}
\int\p_t(\nabla u)\cdot\nabla^2\r dx&=\int\p_t(\nabla u\cdot\nabla^2\r)dx-\int\nabla u\cdot\nabla^2\r_t dx\\
&=\frac{d}{dt}\int\nabla u\cdot\nabla^2\r dx+\int\nabla\dive u\cdot\p_t(\nabla\r)dx\\
&=\frac{d}{dt}\int\nabla u\cdot\nabla^2\r dx+\int\nabla\dive u\cdot(\nabla S_1-\nabla\dive u)dx.
\end{split}
\end{equation*}
Then using H\"{o}lder and Cauchy inequalities, we obtain
\begin{equation*}
\begin{split}
&\frac{d}{dt}\int\nabla u\cdot\nabla^2\r dx+P'(1)\int|\nabla^2\r|^2 dx\\
&=\int[\mu\Delta\nabla u+(\mu+\lambda)\nabla^2\dive u]\cdot\nabla^2\r dx
  +\int\nabla S_2\cdot\nabla^2\r dx\\
&\quad +\int\nabla\dive u\cdot(\nabla\dive u-\nabla S_1)dx\\
&\lm \frac{1}{8} P'(1)\|\nabla^2\r\|^2_{L^2}+C\|\nabla^2u\|^2_{H^1}+C(\|\nabla S_2\|^2_{L^2}+\|S_1\|^2_{L^2}),
\end{split}
\end{equation*}
which, together with \eqref{5} and \eqref{221}, yields directly
\begin{equation*}
\begin{split}
\frac{d}{dt}\int\nabla u\cdot\nabla^2\r dx+\frac{7}{8}P'(1)\int|\nabla^2\r|^2 dx
\lm C\|\nabla^2u\|^2_{H^1}+C(\|\r\|^{\frac14}_{L^2}+\|\r \|_{H^1}+\|u \|_{H^1})\|\nabla^2\r\|^2_{L^2}.
\end{split}
\end{equation*}
Therefore, we complete the proof of this lemma.
\end{proof}

Combining all the estimates obtained in Lemmas \ref{lemm1}-\ref{lemm3},
we drive the following energy estimate.
\begin{lemm}\label{diss}
Under the assumptions of Theorem \ref{Thm1}, we define
\begin{equation*}
\mathcal{E}^2_1(t)\stackrel{\Delta}{=}\|\nabla u\|^2_{H^1}+P'(1)\|\nabla \r\|^2_{H^1}+2\delta_0\int\nabla u\cdot\nabla^2\r dx.
\end{equation*}
Then there exists a large time $T_0$, such that
\begin{equation}\label{1111}
\frac{d}{dt}\mathcal{E}^2_1(t)+c_* (\|\nabla^2 u\|_{H^1}^2+\|\nabla^2 \r\|_{L^2}^2) \lm 0,
\end{equation}
holds on for all $t\ge T_0$.
Here $c_*=\min{\{\mu,\delta_0P'(1)}\}$, and $\delta_0$ is a small constant.
\end{lemm}
\begin{proof}
Adding estimate \eqref{estimate1} with \eqref{estimate2}, it holds on
\begin{equation}\label{241}
\begin{split}
&\frac{d}{dt}\left\{\frac12\|\nabla u\|_{H^1}^2+\frac{P'(1)}{2}\|\nabla\r\|_{H^1}^2\right\}
   +\mu \|\nabla^2 u\|_{H^1}^2\\
&\lm  C(\|\r\|^{\frac14}_{L^2}+\|u\|^{\frac14}_{L^2}+\|\nabla u\|^{\frac14}_{L^2}+\|\r \|_{H^1}+\|u \|_{H^1})
       (\|\nabla^2 u\|^2_{H^1}+\|\nabla^2 \r\|^2_{L^2}).
\end{split}
\end{equation}
Multiplying $\delta_0$ to \eqref{estimate3} and adding with \eqref{241},
we choose $\delta_0$  being small enough to obtain
\begin{equation*}
\begin{split}
&\frac{d}{dt}\left\{\frac12\|\nabla u\|_{H^1}^2+\frac{P'(1)}{2}\|\nabla\r\|_{H^1}^2
  +\delta_0 \int\nabla u\cdot\nabla^2\r dx \right\}\\
& \quad +\frac{3\mu}{4}\|\nabla^2 u\|_{H^1}^2 +\frac{3\delta_0}{4}P'(1)\|\nabla^2 \r\|_{L^2}^2\\
&\lm C(\|\r\|^{\frac14}_{L^2}+\|u\|^{\frac14}_{L^2}+\|\nabla u\|^{\frac14}_{L^2}+\|\r \|_{H^1}+\|u \|_{H^1})
       (\|\nabla^2 u\|^2_{H^1}+\|\nabla^2 \r\|^2_{L^2}).
\end{split}
\end{equation*}
Thanks to the decay rate \eqref{decay} obtained in Theorem \ref{Thm1}, one may conclude that
\begin{equation*}
\|\r\|^{\frac14}_{L^2}+\|u\|^{\frac14}_{L^2}+\|\nabla u\|^{\frac14}_{L^2}+\|\r \|_{H^1}+\|u \|_{H^1}
\le C(1+t)^{-\frac{1}{4}(\frac{2}{p}-1)},
\end{equation*}
and hence, there exists a large time $T_0>0$ such that
$$
C(\|\r\|^{\frac14}_{L^2}+\|u\|^{\frac14}_{L^2}+\|\nabla u\|^{\frac14}_{L^2}+\|\r \|_{H^1}+\|u \|_{H^1})
\le \frac{1}{4}\min{\{\mu,\delta_0P'(1)}\},
$$
holds on for all $t\ge T_0$. Thus, we obtain the energy estimate
\begin{equation*}
 \frac{d}{dt}\left\{\|\nabla u\|_{H^1}^2+{P'(1)}\|\nabla\r\|_{H^1}^2
  +2 \delta_0 \int\nabla u\cdot\nabla^2\r dx \right\}
  +\mu \|\nabla^2 u\|_{H^1}^2 +\delta_0 P'(1)\|\nabla^2 \r\|_{L^2}^2
  \lm 0.
\end{equation*}
Taking $c_*=\min{\{\mu,\delta_0P'(1)}\}$, it holds on
\begin{equation*}
\frac{d}{dt}\mathcal{E}^2_1(t)+c_* (\|\nabla^2 u\|_{H^1}^2+\|\nabla^2 \r\|_{L^2}^2)\lm0.
\end{equation*}
 By virtue of the smallness of $\delta_0$, there are two constants $c_1$ and $C_1$(independent of time) such that
\begin{equation}\label{E}
c_1(\|\nabla u\|^2_{H^1}+\|\nabla \r\|^2_{H^1})\lm\mathcal{E}^2_1(t)\lm C_1(\|\nabla u\|^2_{H^1}+\|\nabla \r\|^2_{H^1}),
\end{equation}
Therefore, we complete the proof of this lemma.
\end{proof}

Finally, let us prove the upper bound of decay for the first and second order spatial
derivatives of global solution to the Cauchy problem \eqref{cns1} with large initial data.
\begin{lemm}\label{lemma25}
Under the assumptions of Theorem \ref{Thm1}, there exists a large time $T_1$, such that
\begin{equation}\label{decay2}
\|\nabla \r (t)\|_{H^1}+\|\nabla u(t)\|_{H^1}
+\|\partial_t \r (t)\|_{L^2}+\|\partial_t u (t)\|_{L^2}
\lm C(1+t)^{-\frac{3}{4}(\frac2p-1)-\frac{1}{2}},
\end{equation}
holds on for all $t\lg T_1$.
Here $C$ is a constant independent of time.
\end{lemm}
\begin{proof}
In order to obtain the time decay rate \eqref{decay2}, our method here is to use the
Fourier splitting method(by Schonbek \cite{schonbek1}), which has been applied to obtain decay rate for the
incompressible Navier-Stokes equations in higher order derivative norm(cf.\cite{{schonbek2},{Schonbek-Wiegner}}).
The difficulty,
arising from the compressible Navier-Stokes equations, is the appearance of density that obeys the
transport equation rather than diffusive one. To get rid of this difficulty, our idea is to rewrite the
inequality \eqref{1111} as follows
\begin{equation}\label{enery}
\frac{d}{dt}\mathcal{E}^2_1(t)+\frac{c_*}{2}\int(|\nabla^2u|^2+|\nabla^3u|^2) dx+\frac{c_*}{2}\int|\nabla^2\r|^2 dx+\frac{c_*}{2}\int|\nabla^2\r|^2 dx\lm0.
\end{equation}
Define $S_0 :=\{\xi\in\R^3||\xi|\lm(\frac{R}{1+t})^{\frac12}\},$ then we can split the phase space $\R^3$ into two time-dependent regions. Here $R$ is a constant defined below.
By routine checking, we can get that
\begin{equation*}
\begin{split}
\int_{\R^3}|\nabla^2u|^2dx
&\lg\int_{\R^3 /S_0}|\xi|^4|\hat{u}|^2d\xi\lg\frac{R}{1+t}\int_{\R^3/ S_0}|\xi|^2|\hat{u}|^2d\xi\\
&=\frac{R}{1+t}\int_{\R^3}|\xi|^2|\hat{u}|^2d\xi-\frac{R}{1+t}\int_{S_0}|\xi|^2|\hat{u}|^2d\xi\\
&\lg \frac{R}{1+t}\int_{\R^3}|\xi|^2|\hat{u}|^2d\xi-\frac{R^2}{(1+t)^2}\int_{S_0}|\hat{u}|^2d\xi,
\end{split}
\end{equation*}
or equivalently
\begin{equation}\label{2511}
\|\nabla^2u\|^2_{L^2}\lg\frac{R}{1+t}\|\nabla u\|^2_{L^2}-\frac{R^2}{(1+t)^2}\|u\|^2_{L^2}.
\end{equation}
In an analogous manner, we ultimately obtain
\begin{equation}\label{252}
\|\nabla^3u\|^2_{L^2}\lg\frac{R}{1+t}\|\nabla^2u\|^2_{L^2}-\frac{R^2}{(1+t)^2}\|\nabla u\|^2_{L^2},
\end{equation}
and
\begin{equation}\label{253}
\|\nabla^2\r\|^2_{L^2}\lg\frac{R}{1+t}\|\nabla\r\|^2_{L^2}-\frac{R^2}{(1+t)^2}\|\r\|^2_{L^2}.
\end{equation}
Substituting the estimates \eqref{2511}-\eqref{253} into \eqref{enery}, one may show that
\begin{equation*}
\frac{d}{dt}\mathcal{E}^2_1(t)
+\frac{c_*}{2}\frac{R}{1+t}(\|\nabla u\|^2_{H^1}+\|\nabla\r \|^2_{L^2})
+\frac{c_*}{2}\|\nabla^2 \r\|^2_{L^2}
\lm \frac{c_* R^2}{2(1+t)^2}(\|u\|^2_{H^1}+\|\r\|^2_{L^2}),
\end{equation*}
holds on for all $t\ge T_0$($T_0$ defined in Lemma \ref{diss}).
Thanks to the equivalent relation \eqref{E}, the term $\|\nabla^2 \r\|^2_{L^2}$
on left handside of the above inequality plays a role of damping term.
And hence, it holds on
\begin{equation*}
\frac{d}{dt}\mathcal{E}^2_1(t)
+\frac{c_* R}{2(1+t)}(\|\nabla u\|^2_{H^1}+\|\nabla\r \|^2_{H^1})
\lm \frac{c_* R^2}{2(1+t)^2}(\|u\|^2_{H^1}+\|\r\|^2_{L^2}),
\end{equation*}
for all $t\ge T_1:=\max \{T_0, R-1\}$.

Thanks to decay estimate \eqref{decay} in Theorem \ref{Thm1} and equivalent relation \eqref{E}, we have
\begin{equation*}
\frac{d}{dt}\mathcal{E}^2_1(t)+\frac{c_*R}{2C_1(1+t)}\mathcal{E}^2_1(t)
\lm \frac{c_* R^2}{2}C(1+t)^{-\frac3p-\frac12}.
\end{equation*}
Choosing $R=\frac{6C_1}{pc_*}$ and multiplying the resulting inequality by $(1+t)^{\frac3p}$, it follows that
\begin{equation*}
\frac{d}{dt}\left[ (1+t)^{\frac3p} \mathcal{E}^2_1(t)\right] \lm C(1+t)^{-\frac12}.
\end{equation*}
For $T_1 {=}\max\{T_0, \frac{6C_1}{pc_*}-1\}$, the integration over $[T_1,t]$ yields directly
\begin{equation*}
\mathcal{E}^2_1(t)
\le (1+t)^{-\frac3p}(1+T_1)^{\frac3p}\mathcal{E}^2_1(T_1)
    +C(1+t)^{-\frac3p}[(1+t)^{\frac{1}{2}}-(1+T_1)^{\frac{1}{2}}],
\end{equation*}
which, together with uniform bound \eqref{uniform-bound} and equivalent relation \eqref{E}, implies
\begin{equation*}
\mathcal{E}^2_1(t)\lm C(1+t)^{-\frac{3}{2}(\frac2p-1)-1}.
\end{equation*}
Using the equivalent relation \eqref{E} again, then it holds on
\begin{equation}\label{251}
\|\nabla\r(t)\|^2_{H^1}+\|\nabla u(t)\|^2_{H^1}\lm C(1+t)^{-\frac{3}{2}(\frac2p-1)-1},
\end{equation}
for all $t \ge T_1 {=}\max\{T_0, \frac{6C_1}{pc_*}-1\}$.

Finally, we establish the decay rate for the time derivative of density and velocity.
Using the first equation of \eqref{cns1}, estimate \eqref{5} and decay rate \eqref{251}, we have
\begin{equation}\label{time1}
\begin{split}
\|\partial_t\r\|_{L^2}& \lm \|\dive u\|_{L^2}+\|S_1\|_{L^2}\\
&\lm \|\nabla u\|_{L^2}+C(\|\r \|_{H^1}+\|u \|_{H^1})(\|\nabla^2 u\|_{L^2}+\|\nabla^2 \r\|_{L^2})\\
&\lm C(1+t)^{-\frac34(\frac2p-1)-\frac12}
     +C(1+t)^{-\frac34(\frac2p-1)}(1+t)^{-\frac34(\frac2p-1)-\frac12}\\
&\lm  C(1+t)^{-\frac34(\frac2p-1)-\frac12}.
\end{split}
\end{equation}
In an analogous fashion, it follows that
\begin{equation}\label{time2}
\begin{split}
\|\partial_t u\|_{L^2}&\lm \mu\|\Delta u\|_{L^2}+(\mu+\lambda)\|\nabla\dive u\|_{L^2}+P'(1)\|\nabla\r\|_{L^2}+\|S_2\|_{L^2}\\
&\lm C(1+t)^{-\frac34(\frac2p-1)-\frac12}.
\end{split}
\end{equation}
The combination of \eqref{251}, \eqref{time1} and \eqref{time2} completes the proof of this lemma.
\end{proof}

\begin{rema}
The decay rate \eqref{decay2} tells us the fact that the first and second order spatial derivatives
of velocity and density converge to zero at the $L^2-$rate $(1+t)^{-\frac{3}{4}(\frac2p-1)-\frac{1}{2}}$
although the initial data $(\rho_0-1, u_0)$ may be large in the sense of $H^2\cap L^p(p\in[1, 2))$ norm.
It should be pointed out that the second order spatial derivative of velocity will converge to zero at the
$L^2-$rate $(1+t)^{-\frac{3}{4}(\frac2p-1)-1}$ for the classical incompressible Navier-Stokes equations
(cf.\cite{Schonbek-Wiegner}).
However, this is still an open problem for the compressible Navier-Stokes equations with large initial data,
or even the small one.
\end{rema}

\subsection{Lower Bound of Decay Rate}\label{lower-bound-decay}

\quad
In this subsection, the content of our analysis is to address the lower bound of decay rate
for the global solution of Cauchy problem \eqref{cns1}.
For the sake of simplicity, we only study the lower bound of decay rate for the global solution
with initial data of the form $(\varrho_0, u_0)\in H^2 \cap L^1$.
Now, we are in a position to prove the lower bounds of decay rates \eqref{decay31} and \eqref{decay32}. Let us define $m:=\rho u$, we rewrite \eqref{cns} in the perturbation form as
\begin{equation}\label{mm}
  	\left\{\begin{aligned}
  &\partial_t\r+\dive m=0,\\
   &\partial_tm-\mu\Delta m-(\mu+\lambda)\nabla\dive m+P'(1)\nabla\r=-\dive F,
     	\end{aligned}\right.
  \end{equation}
where the function $F=F(\r, u)$ is defined as
\begin{equation}
F:=(1+\r)u\otimes u+\mu\nabla(\r u)+(\mu+\lambda)\dive(\r u)\mathbb{I}_{3\times 3}+(P(1+\r)
   -P(1)-P'(1)\r)\mathbb{I}_{3 \times 3}.
\end{equation}
Here the pressure $P(\rho)=\rho^\gamma$ with $\gamma \ge 1$. The initial data is given as
\begin{equation*}
(\r,m)(x,t)|_{t=0}=(\r_0, m_0)(x).
\end{equation*}
In order to obtain the lower decay estimate, we need to analyze the linearized part:
\begin{equation}\label{linear1}
  	\left\{\begin{aligned}
  &\partial_t\r_l+\dive m_l=0,\\
   &\partial_tm_l-\mu\Delta m_l-(\mu+\lambda)\nabla\dive m_l+P'(1)\nabla\r_l=0,
     	\end{aligned}\right.
  \end{equation}
  together with the initial data
  \begin{equation*}
  (\r_l, m_l)(x,t)|_{t=0}=(\r_0, m_0).
  \end{equation*}
  Here the initial data for the linearized part \eqref{linear1} is the same as the nonlinear part \eqref{mm}.
  The following estimates can be found in \cite{{Hu-Wu}, {li-zhang1}}.
  \begin{prop}\label{decay-linear}
Assume the Fourier transform
$\mathcal{F}(\r_0, m_0):=(\widehat{\r_0}, \widehat{m_0})$
satisfies
$|(\widehat{\r_0}, \widehat{m_0})|\le C|\xi|^\eta$ for $0\le |\xi| \ll 1$.
Then, the solution $(\r_l, m_l)$ of linearized system \eqref{linear1}
has the following estimate
\begin{equation}\label{311}
\|(\r_l, m_l)(t)\|_{L^2}
\le C(1+t)^{-(\frac{3}{4}+\frac\eta2)}
    (\|(\widehat{\r_0}, \widehat{m_0})\|_{L^\infty}+\|(\r_0, m_0)\|_{L^2}),
\end{equation}
for all $t \ge 0$.
Furthermore, if the Fourier transform
$\mathcal{F}(\r_0, m_0)=(\widehat{\r_0}, \widehat{m_0})$
satisfies
$$
|\widehat{\r_0}|\ge c_0, ~ \widehat{m_0}=0, ~0 \le |\xi| \ll 1,
$$
with $c_0>0$ a constant, then we have for large time $t$
\begin{equation}\label{312}
\min \{\| \r_l(t)\|_{L^2}, \|m_l(t)\|_{L^2}\}
\ge c_1 (1+t)^{-\frac{3}{4}},
\end{equation}
where $c_1$ and $C$ are positive constants independent of time $t$.
\end{prop}
 Define $\r_\delta:=\r-\r_l$ and $m_\delta:=m-m_l$,
 then $(\r_{\delta}, m_{\delta})$ will satisfy the following equations
  \begin{equation}\label{eq-difference}
  \left\{
  \begin{aligned}
  &\p_t\r_\delta+\dive m_\delta=0,\\
  &\partial_tm_\delta-\mu\Delta m_\delta-(\mu+\lambda)\nabla\dive m_\delta+P'(1)\nabla\r_\delta=-\dive F,
  \end{aligned}
  \right.
  \end{equation}
  with the zero initial data
  \begin{equation}\label{non initial}
  (\r_\delta, m_\delta)(x,t)|_{t=0}=(0,0).
  \end{equation}

Define the differential operator $B$:
\begin{gather}\label{Op-B}
B=
\begin{pmatrix}
0            & -{\rm div}\\
-P'(1)\nabla & \mu\Delta+(\mu+\lambda)\nabla{\rm div}
\end{pmatrix}
\end{gather}
then we can write the solution of \eqref{linear1} and \eqref{eq-difference} as
$$
(\r_l(t), m_l(t))^{tr}=K(t)(\r_0, m_0)^{tr},
$$
and
$$(\r_\delta(t), m_\delta(t))^{tr}=\int_0^tK(t-\tau)(0, -\dive F)^{tr}(\tau)d\tau,$$ respectively,
where $K(t)$ is the solution semigroup defined by $K(t)=e^{tB}, t \ge 0$.
 First of all, let us to establish the upper bound of decay rate for the difference $(\r_\delta, m_\delta)$.
  \begin{lemm}\label{lemma27}
  Let $p=1$, and suppose all the assumptions of Theorem \ref{Thm1} hold on.
  Assume $(\r_\delta, m_\delta)$ be the smooth solution of
  the Cauchy problem \eqref{eq-difference}-\eqref{non initial}.
  Then, it holds on
   \begin{equation}\label{decay-difference}
   \|(\r_\delta, m_\delta)(t)\|_{L^2}
   \le  C(1+t)^{-\frac54},
   \end{equation}
   where $C$ is a constant independent of time.
  \end{lemm}
  \begin{proof}
  By Duhamel principle and estimate \eqref{311}, we have
  \begin{equation}\label{111}
  \begin{split}
 \|(\r_\delta, m_\delta)(t)\|_{L^2}
 &\le \int_0^t(1+t-\tau)^{-\frac54}(\||\xi|^{-1}\mathcal{F}(\dive F)\|_{L^\infty}+\|\dive F\|_{L^2})d\tau \\&
 \le C \int_0^t(1+t-\tau)^{-\frac54}(\|F\|_{L^1}+\|\nabla F\|_{L^2})d\tau.
 \end{split}
  \end{equation}
  By H\"{o}lder and Sobolev inequalities, one may check that
\begin{equation}\label{01}
\|F\|_{L^1}
\le C((1+\|\r\|_{L^\infty})\|u\|_{L^2}^2+\|\nabla\r\|_{L^2}\|u\|_{L^2}
+\|\r\|_{L^2}\|\nabla u\|_{L^2}+\|\r\|^2_{L^2})
\le C\|(\r, u)\|^2_{H^1},
\end{equation}
and
\begin{equation}\label{02}
\begin{split}
\|\nabla F\|_{L^2}&\le C(\|\nabla u\|_{L^3}\|u\|_{L^6}+\|\nabla\r\|_{L^3}\|u\|_{L^\infty}\|u\|_{L^6}+\|u\|_{L^\infty}\|\r\|_{L^6}\|\nabla u\|_{L^3})\\&
\quad +C(\|u\|_{L^\infty}\|\nabla^2\r\|_{L^2}+\|\nabla u\|_{L^\infty}\|\nabla\r\|_{L^2}+\|\r\|_{L^6}\|\nabla^2u\|_{L^3}+\|\r\|_{L^6}\|\nabla\r\|_{L^3})\\&
\le C(\|\nabla u\|_{H^1}\|\nabla u\|_{L^2}+\|\nabla\r\|_{H^1}\|u\|_{H^2}\|\nabla u\|_{L^2}+\|u\|_{H^2}\|\nabla\r\|_{L^2}\|\nabla u\|_{H^1})\\&
 \quad +C(\|\nabla u\|^{\frac12}_{L^2}\|\nabla^2u\|^{\frac12}_{L^2}\|\nabla^2\r\|_{L^2}+\|\nabla u\|^{\frac14}_{L^2}\|\nabla^3u\|^{\frac34}_{L^2}\|\nabla\r\|_{L^2})\\&
 \quad +C(\|\nabla\r\|_{L^2}\|\nabla^2u\|_{H^1}+\|\nabla\r\|_{L^2}\|\nabla\r\|_{H^1})\\&
 \le C(\|(\r,u)\|^2_{H^1}+\|\nabla u\|^{\frac12}_{L^2}\|\nabla^2\r\|_{L^2}+\|\nabla(\r, u)\|_{L^2}(\|\nabla^2u\|_{H^1}+\|\nabla^2\r\|_{L^2})),
\end{split}
\end{equation}
where we have used the uniform estimate \eqref{uniform-bound} and Young inequality in the last inequality.
Then, the combination of \eqref{111}- \eqref{02}  yields immediatelly
  \begin{equation}\label{0010}
  \begin{split}
  \|(\r_\delta, m_\delta)(t)\|_{L^2}
  &\le C\int_0^t(1+t-\tau)^{-\frac54}\|(\r, u)\|^2_{H^1}d\tau
  + C \int_0^t(1+t-\tau)^{-\frac54}\|\nabla u\|^{\frac12}_{L^2}\|\nabla^2\r\|_{L^2} d\tau\\
  &\quad  +C\int_0^t(1+t-\tau)^{-\frac54}\|\nabla(\r, u)\|_{L^2}(\|\nabla^2u\|_{H^1}+\|\nabla^2\r\|_{L^2})d\tau.
    \end{split}
  \end{equation}
 Using H\"{o}lder inequality, we have
 \begin{equation}\label{0011}
 \begin{split}
 &\int_0^t(1+t-\tau)^{-\frac54}\|\nabla(\r, u)\|_{L^2}(\|\nabla^2u\|_{H^1}+\|\nabla^2\r\|_{L^2})d\tau \\
 &\le \left\{\int_0^t(1+t-\tau)^{-\frac52}(1+\tau)^{-2}\|\nabla(\r, u)\|^2_{L^2}d\tau \right\}^{\frac12}
     \left\{\int_0^t(1+\tau)^2(\|\nabla^2u\|^2_{H^1}+\|\nabla^2\r\|^2_{L^2})d\tau \right\}^{\frac12} \\
 &\le C  \left\{\int_0^t(1+t-\tau)^{-\frac52}(1+\tau)^{-2}\|\nabla(\r, u)\|^2_{L^2}d\tau \right\}^{\frac12},
 \end{split}
\end{equation}
 where we have used the estimate
  \begin{equation}\label{113}
  \int_0^t(1+\tau)^2(\|\nabla^2u\|^2_{H^1}+\|\nabla^2\r\|^2_{L^2})d\tau\le C,
    \end{equation}
  here $C$ is a positive constant independent of time. Indeed, recall the inequality \eqref{1111}, we have
\begin{equation*}
\frac{d}{dt}\mathcal{E}^2_1(t)+c_* (\|\nabla^2 u\|_{H^1}^2+\|\nabla^2 \r\|_{L^2}^2) \le 0, \quad t\ge T_0.
\end{equation*}
Multiplying the above inequality by $(1+t)^2$  yields directly
\begin{equation*}
\frac{d}{dt}((1+t)^2\mathcal{E}^2_1(t))
+c_*(1+t)^2 (\|\nabla^2 u\|_{H^1}^2+\|\nabla^2 \r\|_{L^2}^2)
\le 2(1+t)\mathcal{E}^2_1(t),
\end{equation*}
which, together with equivalent relation \eqref{E}
and decay rate \eqref{decay2} with $p=1$, gives directly
\begin{equation*}
\frac{d}{dt}[(1+t)^2\mathcal{E}^2_1(t)]+c_*(1+t)^2(\|\nabla^2 u\|_{H^1}^2+\|\nabla^2 \r\|_{L^2}^2)
\le C(1+t)^{-\frac32}, \quad t\ge T_1.
\end{equation*}
Integrating the above inequality over $[T_1,t]$ and using the uniform estimate \eqref{uniform-bound},
we obtain the estimate \eqref{113}.
Using the same method with \eqref{0011}, we also have
\begin{equation}\label{0012}
\begin{split}
\int_0^t(1+t-\tau)^{-\frac54}\|\nabla u\|^{\frac12}_{L^2}\|\nabla^2\r\|_{L^2} d\tau
\le C(\int_0^t(1+t-\tau)^{-\frac52}(1+\tau)^{-2}\|\nabla u\|_{L^2}d\tau)^{\frac12}.
\end{split}
\end{equation}
This together with \eqref{0010} and \eqref{0011}, and using decay estimate \eqref{decay}, it follows that
\begin{equation}
  \begin{split}
  \|(\r_\delta, m_\delta)(t)\|_{L^2}
  &\le C\int_0^t(1+t-\tau)^{-\frac54}\|(\r, u)\|^2_{H^1}d\tau\\
  &\quad +C \left\{\int_0^t(1+t-\tau)^{-\frac52}(1+\tau)^{-2}\|\nabla(\r, u)\|^2_{L^2}d\tau \right\}^{\frac12}\\
  &\quad +C\left\{\int_0^t(1+t-\tau)^{-\frac52}(1+\tau)^{-2}\|\nabla u\|_{L^2}d\tau\right\}^{\frac12}\\
  &\le C\int_0^t(1+t-\tau)^{-\frac54}(1+\tau)^{-\frac32}d\tau\\
  &\quad +C\left\{\int_0^t(1+t-\tau)^{-\frac52}(1+\tau)^{-\frac{11}{4}}d\tau \right\}^{\frac12}\\&
  \quad +C\left\{\int_0^t(1+t-\tau)^{-\frac52}(1+\tau)^{-\frac72}d\tau \right\}^{\frac12} \\&
   \le C(1+t)^{-\frac54}.
   \end{split}
  \end{equation}
This completes the proof of this lemma.
\end{proof}
Finally, we establish the lower bound of decay rate for the global solution of
compressible Navier-Stokes equations \eqref{cns1}.

  \begin{lemm}\label{lemma28}
  Let $p=1$, and assume all the assumptions of Theorem \ref{Thm1} hold on. Denote $m_0:=\rho_0u_0$, assume that the Fourier transform $\mathcal{F}(\r_0, m_0)=(\widehat{\r_0}, \widehat{m_0})$
  satisfies $|\widehat{\r_0}|\ge c_0$, $\widehat{m_0}=0,0 \le |\xi| \ll 1$, with $c_0>0$ a constant.
  Then, it holds on for large time $t$
  \begin{gather}\label{281}
   c_3(1+t)^{-\frac34}\le\|\r (t)\|_{L^2}\le C_1 (1+t)^{-\frac34};\\
   \label{282}
   c_3(1+t)^{-\frac34}\le\| u(t)\|_{L^2}\le C_1 (1+t)^{-\frac34}.
  \end{gather}
  Here $c_3$ and $C_1$ are constants independent of time.
  \end{lemm}
  \begin{proof}
 The upper bounds of decay rates \eqref{281} and \eqref{282} have been given
 in estimate \eqref{decay} for the case $p=1$.
 In the sequel, we will establish the lower bounds of decay rates in \eqref{281} and \eqref{282}.
 Remember the definition
 $\r_\delta:=\r-\r_l$ and $m_\delta:=m-m_l$, then it holds on
 \begin{equation*}
 \|\r_l\|_{L^2}\le\|\r\|_{L^2}+\|\r_\delta\|_{L^2}, \quad  \|m_l\|_{L^2}\le\|m\|_{L^2}+\|m_\delta\|_{L^2},
 \end{equation*}
 which, together with lower bound decay \eqref{312} and upper bound decay \eqref{decay-difference}, yields
 \begin{equation}\label{1112}
 \|\r(t)\|_{L^2} \lg \|\r_l(t)\|_{L^2}-\|\r_\delta(t)\|_{L^2}
 \lg c_1(1+t)^{-\frac34}-C(1+t)^{-\frac54}\lg c_2(1+t)^{-\frac34},
 \end{equation}
 and
 \begin{equation*}
 \|m(t)\|_{L^2} \lg \|m_l(t)\|_{L^2}-\|m_\delta(t)\|_{L^2}
 \lg c_1(1+t)^{-\frac34}-C(1+t)^{-\frac54}\lg c_2(1+t)^{-\frac34},
 \end{equation*}
 for large time $t$.
 Recall that $m:=\rho u$, then by using decay estimate \eqref{decay}, we have
 \begin{equation*}
 \|m(t)\|_{L^2}\le\|u(t)\|_{L^2}+\|\r(t)\|_{L^3}\|u(t)\|_{L^6}
               \le \|u(t)\|_{L^2}+C\|\r(t)\|_{H^1}\|\nabla u(t)\|_{L^2}\le \|u(t)\|_{L^2}+C(1+t)^{-\frac32}.
 \end{equation*}
For large time $t$, it follows that
 \begin{equation}\label{1113}
 \|u(t)\|_{L^2}
 \lg\|m(t)\|_{L^2}-C(1+t)^{-\frac32}\lg c_2(1+t)^{-\frac34}-\frac{C}{(1+t)^{\frac34}}(1+t)^{-\frac34}
\lg c_3(1+t)^{-\frac34}.
 \end{equation}
 The combination of \eqref{1112} and \eqref{1113} yields directly
 \begin{equation*}
 \min\{\|\r(t)\|_{L^2},\|u(t)\|_{L^2}\}\lg c_3(1+t)^{-\frac34}.
 \end{equation*}
Therefore, we complete the proof of this lemma.
\end{proof}
\section*{Acknowledgements}

The authors would like to thank Dr. Jingchi Huang sincerely for pointing out that it is maybe possible to improve their decay result in \cite{he-huang-wang} and also for his valuable suggestion. Jincheng Gao's research was partially supported by
Fundamental Research Funds for the Central Universities of China(Grant No.18lgpy66)
and NSFC(Grant Nos.11571380 and 11801586).
Zhengzhen Wei's research was partially supported by NSFC(Grant No.11701585).
Zheng-an Yao's research was partially supported by NSFC(Grant Nos.11431015 and 11971496).

\phantomsection

\end{document}